\documentclass[11pt]{article}
\usepackage{amsmath,amssymb,enumerate,color}
\usepackage{amsthm}

\newcommand{\R}{\mathbb{R}}
\newcommand{\C}{\mathbb{C}}
\newcommand{\Z}{\mathbb{Z}}
\newcommand{\N}{\mathbb{N}}
\newcommand{\ep}{\varepsilon}
\newcommand{\pa}{\partial}

\newtheorem{theorem}{Theorem}[section]
\newtheorem{lemma}[theorem]{Lemma}
\newtheorem{proposition}[theorem]{Proposition}

\theoremstyle{remark}
\newtheorem{remark}{Remark}[section]
\theoremstyle{definition}
\newtheorem{definition}{Definition}

\numberwithin{equation}{section}

\makeatletter
\def\@cite#1#2{[{{\bfseries #1}\if@tempswa , #2\fi}]}
\makeatother                  
\begin{document}

\begin{center}
\Large{{\bf
Spectral properties of non-selfadjoint extensions of 
\\
the Calogero Hamiltonian 
}}
\end{center}

\vspace{5pt}

\begin{center}
Giorgio Metafune%
\footnote{Dipartimento di Matematica ``Ennio De Giorgi'', 
Universit\`a del Salento, Via Per Arnesano, 73100, Lecce, Italy, 
E-mail:\ {\tt giorgio.metafune@unisalento.it}}
and Motohiro Sobajima%
\footnote{
Dipartimento di Matematica ``Ennio De Giorgi'', 
Universit\`a del Salento, Via Per Arnesano, 73100, Lecce, Italy, 
E-mail:\ {\tt msobajima1984@gmail.com}}
\end{center}

\newenvironment{summary}{\vspace{.5\baselineskip}\begin{list}{}{%
     \setlength{\baselineskip}{0.85\baselineskip}
     \setlength{\topsep}{0pt}
     \setlength{\leftmargin}{12mm}
     \setlength{\rightmargin}{12mm}
     \setlength{\listparindent}{0mm}
     \setlength{\itemindent}{\listparindent}
     \setlength{\parsep}{0pt}
     \item\relax}}{\end{list}\vspace{.5\baselineskip}}
\begin{summary}
{\footnotesize {\bf Abstract.}
We describe all extensions of the Calogero Hamiltonian \[
L=-\frac{d^2}{dr^2}+\frac{b}{r^2} \quad \text{in}\ L^2(\R_+), \quad b <-\frac14
\]  having non empty resolvent and generating an analytic semigroup in $L^2(\R_+)$.
}
\end{summary}




\section{Introduction}

We study spectral properties of the Calogero Hamiltonian in $\R_+:=(0,\infty)$, 
that is of 
one-dimensional Schr\"odinger operator with inverse square potentials 
\[
L=-\frac{d^2}{dr^2}+\frac{b}{r^2} \quad \text{in}\ L^2(\R_+), 
\]
where $b\in \R$. 
By  Hardy's inequality the quadratic form 
\[
\int_0^\infty \left (| u'(r)|^2+\frac{b}{r^2}|u(r)|^2\right )\, dr
\] 
is nonnegative on $D(L_{\min}):=C_0^\infty(\R_+)$ 
if and only if $b \ge -\frac{1}{4}$. 
In this case the Friedrichs extension of $L_{min}$ 
is selfadjoint and nonegative. 
Moreover, if $b \ge \frac{3}{4}$, then $L_{\min}$ is essentially selfadjoint. 
In $N$-dimensional case, the threshold for nonnegativity of 
\[
\int_{\R^N} \left(| \nabla u(x)|^2+\frac{b}{|x|^2}|u(x)|^2\right )\, dx, 
\qquad u\in D(L_{\min}):=C_c^\infty(\R^N\setminus\{0\})
\] 
is $-(\frac{N-2}{2})^2$ 
and that for the  essentially selfadjointness of 
$L_{\min}=-\Delta +b|x|^{-2}$ is $-(\frac{N-2}{2})^2+1$.
These constants are the optimal constants of Hardy's and Rellich's
inequalities, respectively,  see \cite{RS2}. 
On the other hand if $b<-\left(\frac{N-2}{2}\right)^2$, 
Baras and Goldstein  proved in \cite{BG84} that 
there is no positive distributional solution of the equation 
\begin{equation}\label{BG.eq}
\ u_t(x,t)-\Delta u(x,t)+\frac{b}{|x|^2}u(x,t)=0, \qquad (x,t)\in \R^N\times \R_+
\end{equation}
apart from the zero solution. 
This nonexistence result for positive solutions has been generalized 
by subsequent papers (\cite{CM99}, \cite{GJ01}, \cite{GJ02}, \cite{Ma03} and \cite{Ga08}). 
Since for $b<-\left(\frac{N-2}{2}\right)^2$ 
the quadratic form above is unbounded from below,  every 
selfadjoint extension of $L_{\min}$ has a spectrum unbounded from below 
and cannot be the (minus) generator of a semigroup.

\noindent In this paper we mainly consider the one dimensional case and assume that
\begin{equation}\label{nu}
b<-\frac{1}{4} \quad \nu:=\sqrt{-\frac{1}{4}-b} >0. 
\end{equation}
We characterize all intermediate operators between $L_{\min}$ and $L_{\max}:=(L_{\min})^*$, 
given by
\[
   D(L_{\max})
   :=
   \{u\in L^2(\R_+)\cap H^2_{\rm loc}(\R_+)
      \;;\; Lu\in L^2(\R_+)
   \}, 
\]
with non-empty resolvent set, 
including all selfadjoint extensions, and describe their spectrum. Spectral properties of selfadjoint extensions 
are also considered in \cite{GTV10} when $b<-\frac{1}{4}$. 
We show that there exist infinitely many non-selfadjoint extensions $-\tilde L$ 
which are generators of analytic semigroups. Since Hardy's inequality fails, these semigroups cannot be (quasi) contractive.
Some partial results in the $N$-dimensional case are stated in the last section.
Vazquez and Zuazua pointed out in \cite{VZ00} that 
the existence of solutions of \eqref{BG.eq} might require a lower bound of $b$ 
and a restriction of initial data. 
Our result, in contrast, are valid for any $b\in\R$ and any initial datum in $L^2(\R^N)$.

%
%

\section{Preliminaries}
In this section we study the equation $\lambda u+Lu=f$. 

\subsection{The homogeneous equation} 
If $\lambda \not\in ]-\infty,0]$ the above equation with $f=0$ has two solutions, 
one exponential decaying, the other exponential growing at $\infty$. 
The behavior of these two solutions near $0$ is studied in the next two lemmas. To state them, for $\nu>0$ we  define
\begin{equation} \label{defalpha}
\alpha=\alpha(\nu)=c\frac{2^{-i\nu}}{\nu\Gamma(i\nu)}= c\frac{2^{-i\nu}i}{\Gamma(1+i\nu)},
\end{equation}
where $c>0$ is independent of $\nu$ and will play no role in what follows.
\begin{lemma}\label{behavior}
Let $\omega\in\C_+$,  
$\omega=\mu e^{i\xi}$ with $\mu>0$, $|\xi| <\pi/2$ and
assume that \eqref{nu} holds. Then there exists a solution $\varphi_{\omega,0}$ of
\begin{equation}\label{res.eq}
\omega^2 \varphi(r) - \varphi''(r)+\frac{b}{r^2}\varphi(r)=0, \quad r\in \R_+
\end{equation}
and a constant $R=R(b,\omega)>0$ such that 
\begin{gather}
\label{phi0.inf}
\left|\varphi_{\omega,0}(r)\right| \leq 2e^{-({\rm Re}\,\omega)r}, 
\quad r\geq R.
\end{gather}
Moreover $\varphi_{\omega,0}(r)$ is real when $\omega$ is real and
\begin{gather}
\label{phi0.zero}
\left|
r^{-\frac{1}{2}}\varphi_{\omega,0}(r)
-\mu^{\frac{1}{2}}e^{i\frac{\xi}{2}}
\left(
  \alpha\mu^{i\nu} e^{-\xi\nu}r^{i\nu}
  +
  \overline{\alpha}\mu^{-i\nu}e^{\xi\nu}r^{-i\nu}
\right) 
\right|\to 0
\quad \text{as}\ r\downarrow 0,
\end{gather}
where $\alpha$ is defined in (\ref{defalpha}). 
\end{lemma}

\begin{proof}
{\bf (Step 1).}\ \ 
We consider the modified Bessel equation 
\begin{equation}\label{eq.w}
w(z)-\frac{d^2w}{dz^2}(z)-\frac{b}{z^2}w(z)=0, \quad z\in \C_+.
\end{equation}
The indicial equation $\alpha(\alpha-1) =b$ has roots 
$\alpha_{1}=\frac{1}{2} + i \sqrt \nu$ and $\alpha_{2}=\frac{1}{2} - i \sqrt \nu$. 
Then every solution has the form 
\begin{equation} \label{representation}
w(z)=g_1(z)z^{\frac{1}{2}+i\nu}+g_2(z)z^{\frac{1}{2}-i\nu},
\end{equation}
with $g_1, g_2$ entire functions, $g_1(0) \neq 0, g_2(0) \neq 0$, 
and therefore is holomorphic in $\C\setminus ]-\infty,0]$, see \cite[Chapter 9.6, 9.8]{BR}. 

\noindent Let us show that there exists a solution of \eqref{eq.w} which behaves like 
$e^{-z}$ in $E_R:=\{z\in \C_+\;;\; |z|> R\}$. 
Setting $h(z):=e^{z}w(z)$  
\eqref{eq.w} reduces to
\begin{equation}\label{eq.h}
\frac{d^2h}{dz^2}(z)-2\frac{dh}{dz}(z)=\frac{b}{z^2}h(z), \quad z\in \C_+.
\end{equation}
We indicate with $X:=H^\infty(E_R)$,  the set of all bounded holomorphic functions in $E_R$,
endowed with $\|h\|_X:=\sup_{z\in E_R}|h(z)|$. 
Define 
\begin{align}
Th(z):=
1+\int_{\Gamma_z}
  e^{2\xi}
  \left(\int_{\Gamma_\xi}
    \frac{be^{-2\eta}}{\eta^2}h(\eta)
  \,d\eta\right)
\,d\xi, \quad z\in E_R, 
\end{align}
where $\Gamma_z:=\{tz\;;\;t\in [1,\infty)\}$; note that a fixed point of $T$ satisfies \eqref{eq.h}. 
Then $T: X\to X$ is well-defined and contractive in $X$ when $R$ is large enough.
In fact, if $h\in X$, then  $Th$ is well-defined and holomorphic in $E_R$. 
Moreover, for $z\in E_R$, 
\begin{align*}
|Th(z)-1|
&=\left|\int_{1}^{\infty}
  e^{2t z}
  \left(\int_{t}^{\infty}
    \frac{be^{-2sz}}{(sz)^2}h(sz)
  z\,ds\right)
z\,dt\right|
=\left|\int_{1}^{\infty}
  \left(\int_{1}^{s}
  e^{2t z}
  \,dt\right)
    \frac{be^{-2sz}}{s^2}h(sz)
\,ds\right| 
\\
&\leq
\left|\frac{b(1-e^{2(s-1)z})}{2z}\right|
\left(\int_{1}^{\infty}
    \frac{1}{s^2}\,ds\right)\|h\|_X
\leq
\frac{|b|}{R}\|h\|_X.
\end{align*}
Similarly, we have $|Th_1(z)-Th_2(z)|\leq (|b|/R)\|h_1-h_2\|_X$ 
for every $h_1, h_2\in X$ and $z\in E_R$. 
Therefore $T:X\to X$ is well-defined and 
if we choose $R_0:=2|b|$, then $T$ is contractive.  
Let  $h_0\in X$ be the unique fixed point of $T$. 
Noting that 
\begin{align*}
|h_0(z)-1|=|Th_0(z)-T0(z)|\leq \frac{|b|}{R_0}\|h_0\|_X\leq \frac{\|h_0-1\|_X+1}{2}, 
\end{align*}
we deduce $\|h_0-1\|_X\leq 1$. Taking $w_0(z):=e^{-z}h_0(z)$
it follows that $w_0$ can be continued as a solution of \eqref{eq.w}  and
\begin{align}
|e^{z}w_0(z)|\leq 2, \quad z\in E_{R_0}.
\end{align}
Now we define 
\[
\varphi_{\omega, 0}(r):=w_0(\omega r), \quad r\in \R_+.
\] 
Then $\varphi_{\omega,0}$ solves \eqref{res.eq}
\begin{align*}
\omega^2\varphi_{\omega, 0}(r)-\varphi_{\omega, 0}''(r)
+\frac{b}{r^2}\varphi_{\omega, 0}(r)
&=
\omega^2\left(w_0(\omega r)-\frac{d^2w_0}{dz^2}(\omega r)
+\frac{b}{(\omega r)^2}w_0(\omega r)\right)
=0.
\end{align*}
Moreover, if $r>R:=R_0/|\omega|$, then 
\[
|e^{\omega r}\varphi_{\omega, 0}(r)|
=|e^{\omega r}w_0(\omega r)|\leq 2
\]
and \eqref{phi0.inf} is satisfied.
\medskip

\noindent{\bf (Step 2).}\ \ Next we consider $w_0$ on the positive real axis and we may assume that $w_0$ is real on it  (otherwise we consider $\frac12 (w_0(z)+\overline{w_0}(\overline{z})$). 
By (\ref{representation}) we have 
\begin{equation} \label{representation2}
w_0(z)=g_1(z)z^{\frac{1}{2}+i\nu}+g_2(z)z^{\frac{1}{2}-i\nu}, \quad z\in \C \setminus ]-\infty,0]
\end{equation}
 where $g_1, g_2$ are entire functions. 
Then  $g_1(r)=\overline{g_2(r)}$ for $r>0$  and $\alpha=g_1(0)=\overline{g_2(0)}\neq 0$. 
This implies that 
\[
\left|z^{-\frac{1}{2}}w_0(z)-
\left(\alpha z^{i\nu}+\overline{\alpha}z^{-i\nu}\right)\right|\to 0
\quad \text{as}\ z\to 0 \quad (z\in \C_+).
\]
Consequently we obtain \eqref{phi0.zero} with 
$K_{\omega,0}=\omega^{\frac{1}{2}}=\mu^{\frac{1}{2}}e^{i\frac{\xi}{2}}$
\begin{align*}
&\left|r^{-\frac{1}{2}}\varphi_{\omega,0}(r)-
\mu^{\frac{1}{2}}e^{i\frac{\xi}{2}}\left(\alpha e^{-\xi \nu}\mu^{i\nu}r^{i\nu}
+\overline{\alpha}e^{\xi \nu}\mu^{-i\nu}r^{-i\nu}\right)\right| \\&
=
\mu^{\frac{1}{2}} \left|(\omega r)^{-\frac{1}{2}}w_{0}(\omega r)
-\left(\alpha(\omega r)^{i\nu}+\overline{\alpha}(\omega r)^{-i\nu}\right)\right|
\to 0
\quad \text{as}\ r\downarrow 0.
\end{align*}

\noindent{\bf (Step 3).} Finally we show that $\alpha$ is given by (\ref{defalpha}). 
In fact $\varphi_{1,0}(r)$, being the unique (up to constants) 
exponentially decaying solution of  (\ref{res.eq}) with $\omega=1$, coincides with  
$c r^{\frac{1}{2}}K_{i\nu}(r)$, where $c>0$ and $K_{i\nu}$ is the modified Bessel function of 
second kind. Therefore by \cite[9.6.2 and 9.6.7 in p. 375]{AS} 
we deduce that   
\[
r^{-\frac{1}{2}}\varphi_{1,0}(r)= 
\frac{c\pi(I_{-i\nu}( r)-I_{i\nu}( r))}{2\sin (i\nu \pi)}
\sim 
c'
\left(\frac{2^{-i\nu}}{\nu\Gamma(i\nu)}r^{i\nu}
+
\frac{2^{i\nu}}{\nu\Gamma(-i\nu)}r^{-i\nu}
\right)
\]
as $r\downarrow0$ for some $c'>0$. Therefore $\alpha$ in given by (\ref{defalpha}).
\end{proof}

\noindent Next we investigate the behavior at $0$ of the exponentially growing solution.
\begin{lemma}\label{behavior2}
Let $\omega\in\C_+$ satisfy $\omega=\mu e^{i\xi}$ with $\mu>0$, $|\xi| <\pi/2$ and
assume that \eqref{nu} holds.
Then there exist a solution $\varphi_{\omega,1}$ of \eqref{res.eq}
and constants $C_\omega'>C_\omega>0$ and $R'>0$ such that 
\begin{gather}
\label{phi1.inf}
 C_\omega e^{({\rm Re}\,\omega)r}
 \leq |\varphi_{\omega,1}(r)|\leq C_\omega' e^{({\rm Re}\,\omega)r}
 \quad \text{as}\ r\geq R',
\\
\label{phi1.zero}
\left|
r^{-\frac{1}{2}}\varphi_{\omega,1}(r)
-\mu^{\frac{1}{2}}e^{i\frac{\xi}{2}}
\left(
  \alpha\mu^{i\nu} e^{-\xi\nu}r^{i\nu}
  -
  \overline{\alpha}\mu^{-i\nu} e^{\xi\nu}r^{-i\nu}\right) 
\right|\to 0
\quad \text{as}\ r\downarrow 0, 
\end{gather}
where $\alpha$ is defined in (\ref{defalpha}). Finally,  $i \varphi_{\omega,1}(r)$ is real when $\omega $ is real.
\end{lemma}

\begin{proof}
By (\ref{representation}) 
there exist two solutions $w_1, w_2$  satisfying 
\begin{gather*}
z^{-\frac{1}{2} -i\nu}w_1(z)\to 1, \quad
z^{-\frac{1}{2} + i\nu}w_2(z)\to 1 
\quad \text{as}\ z \to 0.
\end{gather*}
With the notation of the proof of Lemma \ref{behavior} 
we have $\varphi_{\omega,0}(r)=w_0(\omega r)$ 
and $w_0(z)$ is given by (\ref{representation2}), $g_1(r)=\overline{g_2(r)}$ for $r>0$ 
and $\alpha=g_1(0)=\overline{g_2(0)}\neq 0$. 
We take now $v(z)=g_1(z)z^{\frac{1}{2}+i\nu}-g_2(z)z^{\frac{1}{2}-i\nu}$. 
Then $w_0$, $v$ are linearly independent 
and $\varphi_{1,\omega}(r)=v(r\omega)$ is a solution of (\ref{res.eq}) which satisfies (\ref{phi1.zero}), 
by construction and is purely imaginary when $\omega$ is real. 
To prove (\ref{phi1.inf}) we note that (\ref{res.eq}) has one solution 
which behaves like $\exp(-\omega r)$ (namely, $\varphi_{0,\omega}$) 
and one solution which behaves like $\exp(\omega r)$ at $\infty$, see \cite[Proposition 4]{MS_bhv} 
for an elementary proof. 
Since $\varphi_{1,\omega}$ is independent of $\varphi_{0,\omega}$, then (\ref{phi1.inf}) holds.
\end{proof}

\noindent Finally we consider the case where $\omega=i\mu$.

\begin{lemma}\label{behavior3}
Assume that \eqref{nu} holds. Then for every $\mu>0$, there exist two solutions 
$\varphi_{i\mu,0}$ and $\varphi_{i\mu,1}$ of
\begin{equation}\label{res.eq2}
-\mu^2 \varphi(r) - \varphi''(r)+\frac{b}{r^2}\varphi(r)=0, \quad r\in \R_+
\end{equation}
satisfying as $r\to \infty$, 
\begin{gather}
e^{-i\mu r}\varphi_{i\mu,0}(r)\to 1, \qquad e^{i\mu r}\varphi_{i\mu,0}(r)\to i\mu, 
\\
e^{i\mu r}\varphi_{i\mu,1}(r)\to 1, \qquad e^{i\mu r}\varphi_{i\mu,1}'(r)\to -i\mu.
\end{gather}
\end{lemma}

\begin{proof}
It suffices to apply \cite[Proposition 5]{MS_bhv}, with $f(x)=-\mu^2$, to \eqref{res.eq2}.
\end{proof}

\subsection{The inhomogeneous equation}

\begin{lemma}\label{green}
Let $\omega\in\C_+$ satisfy $\omega=\mu e^{i\xi}$ with $\mu>0$, $|\xi| <\pi/2$ and
assume that \eqref{nu} holds. 
Let $\varphi_{\omega,0}$ and $\varphi_{\omega,1}$ be as 
in Lemmas \ref{behavior}, \ref{behavior2} and \ref{behavior3}. Then for $f\in L^2(\R_+)$, 
every solution of 
\begin{align}\label{eq.u.f}
\omega^2 u(r)-u''(r)+\frac{b}{r^2}u(r)=f(r), \quad r\in \R_+
\end{align}
is given by 
\begin{equation} \label{res?}
u(r)=c_0\varphi_{\omega,0}(r)+c_1\varphi_{\omega,1}(r) +T_\omega (f)
\end{equation}
where
\begin{equation} \label{Tomega}
T_\omega (f)(r)=\frac{1}{W(\omega)}\left(\int_{0}^r \varphi_{\omega,1}(s)f(s)\,ds\right)\varphi_{\omega,0}(r)+
\frac{1}{W(\omega)}\left(\int_{r}^\infty \varphi_{\omega,0}(s)f(s)\,ds\right)\varphi_{\omega,1}(r),
\end{equation}
$c_0, c_1\in \C$ are constants and 
$W(\omega)$ is the Wronskian of $\varphi_{\omega,0}, \varphi_{\omega,1}$. 
The map $T_\omega$ is 
a bounded linear operator from $L^2(\R_+)$ to itself and, if $\omega$ is real, $T_\omega$ is selfadjoint.
\end{lemma}

\begin{proof}
By variation of parameters \eqref{res?} easily follows.
Observe that 
\begin{equation*} 
T_\omega f(r)=\int_0^\infty G_{\omega}(r,s)f(s)\, ds, 
\end{equation*}
where 
\begin{equation} \label{defG}
G_\omega(r,s)=\left\{
\begin{array}{l}
W(\omega)^{-1}\varphi_{\omega,0}(r)\varphi_{\omega,1}(s) \quad \text{if}\ s \le r, \\
W(\omega)^{-1}\varphi_{\omega,0}(s)\varphi_{\omega,1}(r) \quad \text{if}\ s \ge r.
\end{array}\right.
\end{equation}
Using Lemmas \ref{behavior}, \ref{behavior2} and recalling that 
both solutions are bounded near $0$ we obtain 
$|\varphi_{\omega,0}(r)| \le C e^{-({\rm Re} \omega) r}$, 
$|\varphi_{\omega,1}(r)| \le C e^{({\rm Re} \omega) r}$ for every $ r>0$. 
Therefore  
\[
|G_\omega(r,s)| \le C^2 e^{-({\rm Re}\, \omega) |r-s|}, \quad r>0, \ s>0
\] 
and 
the boundedness of $T_\omega$ follows. If $\omega$ is real, 
then $\varphi_{\omega,0}, i\varphi_{\omega,1}, iW(\omega)$ are real so that 
$\overline{G_\omega(r,s)}=G_\omega(s,r)$ and $T_\omega$ is selfadjoint. 
\end{proof}

\section{Intermediate operators and their spectral properties}
Here we characterize all extensions $L_{min} \subset \tilde L \subset L_{max}$ with non-empty resolvent set and study their spectral properties.

\begin{lemma}\label{easy}
Let the operator $\tilde{L}$ satisfy $L_{\min}\subset \tilde{L}\subset L_{\max}$. 
Then $[0,\infty)\subset \sigma(\tilde{L})$.
\end{lemma}
\begin{proof}
First we prove $(0,\infty)\in \sigma(\tilde{L})$. 
Let $ \eta_n(r)$ be a smooth function equal to $1$ in $[n,2n]$, 
with support contained in $[n/2,3n]$ and $0 \le \eta_n \le 1$, $|\eta'_n| \le C/2$, $|\eta''_n| \le C/n^2$. 
Given $\varphi_{i\mu,0}$ as in  Lemma \ref{behavior3} 
we consider $\psi_n=\eta_n \varphi_{i\mu,0}\in C_0^\infty(\R_+)\subset D(\tilde{L})$. 
Then $-\mu^2 \psi_n+L\psi_n=-2\eta'_n \varphi'_{i\mu,0}-\eta''_n \varphi_{i\mu,0}$. 
We have $\|\psi_n\|_2 \approx \sqrt n$ and, since $\varphi_{i\mu,0}$ 
has first and second derivatives bounded near $\infty$, $\|(-\mu^2+L)\psi_n\|_2 \le Cn^{-1/2}$. 
Therefore $\mu^2$ is an approximate point spectrum, in other words, 
$-\mu^2+L$ cannot have a bounded inverse.
Finally, noting that $\sigma(\tilde{L})$ is closed in $\C$, 
we have $[0,\infty)\subset \sigma(\tilde{L})$.
\end{proof}

\begin{lemma}\label{intermediate}
Let $L_{\min}\subset \tilde{L}\subset L_{\max}$ and
assume that \eqref{nu} and $\rho(\tilde{L})\neq \emptyset$ hold. Then 
there exists $c \in \C$ such that defining  $(a_1, a_2)\in\C^2\setminus\{(0,0)\}$ by
\begin{equation} \label{a1a2}
a_1=(c+W(\omega)^{-1})\alpha\mu^{i\nu}e^{-\xi \nu} \qquad 
a_2=(c-W(\omega)^{-1})\overline{\alpha}\mu^{-i\nu}e^{\xi \nu}
\end{equation} 
the domain of $\tilde L$ is given by
\begin{equation}\label{characterize}
D(\tilde{L})=\left\{u\in D(L_{\max})\;;\;\exists C\in\C\ \text{s.t.}\ 
\displaystyle
\lim_{r\downarrow 0}\left|r^{-\frac{1}{2}}u(r)
-
C\left(a_1r^{i\nu}+a_2r^{-i\nu}\right)
\right|=0\right\}. 
\end{equation}
\end{lemma}

\begin{proof}
First we show the inclusion ``$\,\subset\,$'' in \eqref{characterize}.
Since, by Lemma \ref{easy} $[0,\infty[\subset\sigma(\tilde{L})$,  
we take $\lambda\in \rho(\tilde{L})$ for some $\lambda\in \C\setminus [0,\infty)$. 
Let $\omega\in \C_+$ satisfy $-\omega^2=\lambda$. 
From Lemma \ref{green}, see (\ref{res?}),  we have 
\begin{align} 
[(\omega^2+\tilde{L})^{-1}f](r)
&=
c_0(f)\varphi_{\omega,0}(r)+c_1(f)\varphi_{\omega,1}(r)+T_\omega f(r). \label{pre.res}
\end{align}
However, $\varphi_{\omega,1}\notin L^2(\R_+)$ and  $\varphi_{\omega,0}\in L^2(\R_+)$. Therefore $c_1(f)=0$
and  $c_0(f)$ 
is a bounded linear functional in $L^2(\R_+)$.
Riesz's representation theorem  yields $v\in L^2(\R_+)$ 
such that 
\begin{equation} \label{defv}
c_0(f)=\int_0^\infty f(s)v(s)\,ds.
\end{equation}
If we choose $f=\omega^2 u+Lu$ for $u\in C_0^\infty(\R_+)$, 
then for $r$ small enough, we see integrating by parts that 
\begin{align}
\nonumber
0=u(r)
&=
c_0(f)\varphi_{\omega,0}(r)
+\frac{1}{W(\omega)}
\left(\int_{0}^\infty \varphi_{\omega,0}(s)f(s)\,ds\right)\varphi_{\omega,1}(r)
=
c_0(f)\varphi_{\omega,0}(r).
\end{align}
Thus $c_0(f)=0$ for every $f\in (\omega^2+L)(C_0^\infty(\R_+))$. 
This yields that $(\omega^2+L)v=0$ and hence 
\begin{equation} \label{defc}
v=c\varphi_{\omega,0}, \qquad c_0(f)=c\int_0^\infty f(s)\varphi_{\omega,0}(s)\,ds \quad {\rm for\  some\  }c\in \C,
\end{equation}
 since $v \in L^2(\R_+)$. 
Consequently, 
for every $f\in L^2(\R_+)$, $u=(\omega^2+\tilde{L})^{-1}f$ satisfies 
\begin{equation} \label{limit}
\lim_{r\downarrow0}r^{-\frac{1}{2}}\left|u(r)-
\left(\int_0^\infty \varphi_{\omega,0}(s)f(s)\,ds\right)
\left(c\varphi_{\omega,0}(r)+W(\omega)^{-1}\varphi_{\omega,1}(r)\right)
\right|=0.
\end{equation}
Using \eqref{phi0.zero} and \eqref{phi1.zero} (with the same notation), 
we obtain ``$\,\subset\,$'' with  $(a_1,a_2) \neq (0,0)$ given by (\ref{a1a2}) and $c$ given by (\ref{defc}). 

\noindent Conversely, we prove the inclusion ``$\,\supset\,$'' 
in \eqref{characterize}.
Let $u\in D(L_{\max})$ satisfy 
\[
\lim_{r\downarrow 0}\left|r^{-\frac{1}{2}}u(r)
-
C'\left(a_1r^{i\nu}+a_2r^{-i\nu}\right)
\right|=0,
\]
where the pair $(a_1, a_2)$ is defined in (\ref{a1a2}) and $c$ in (\ref{defc}).
By  \eqref{phi0.zero} and \eqref{phi1.zero} we have 
\[
\lim_{r\downarrow0}r^{-\frac{1}{2}}\left|u(r)-
C\left(c\varphi_{\omega,0}(r)+W(\omega)^{-1}\varphi_{\omega,1}(r)\right)
\right|=0.
\]
Set $\tilde{u}:=(\omega^2+\tilde{L})^{-1}(\omega^2+L_{\max})u$ 
and $w:=u-\tilde{u}$. 
Then $(\omega^2+L)w=0$ and, since $w \in L^2(\R_+)$, 
$w=c'\varphi_{\omega,0}$ for some $c'\in \C$. 
Noting that 
\[
\lim_{r\downarrow 0}
r^{-\frac{1}{2}}\left|\tilde{u}(r)
-
C'\left(c\varphi_{\omega,0}(r)+W(\omega)^{-1}\varphi_{\omega,1}(r)\right)
\right|=0,
\]
we obtain 
\[
\lim_{r\downarrow0}
r^{-\frac{1}{2}}\left|
c'\varphi_{\omega,0}(r)-(C-C')\left(c\varphi_{\omega,0}(r)+W(\omega)^{-1}\varphi_{\omega,1}(r)\right)
\right|=0
\]
or
\[
\lim_{r\downarrow0}
r^{-\frac{1}{2}}\left|
\left (c'-c(C-C')\right )\varphi_{\omega,0}(r)-(C-C')W(\omega)^{-1}\varphi_{\omega,1}(r)
\right|=0.
\]
By  \eqref{phi0.zero} and \eqref{phi1.zero} again we deduce that $c'=0$, 
hence  $u=\tilde{u}\in D(\tilde{L})$.
\end{proof}

\noindent 
In view of Lemma \ref{intermediate}, 
we define intermediate operators between $L_{\min}$ and $L_{\max}$ as follows. 
\begin{definition}\label{LA}
Let $A:=(a_1,a_2)\in \C^2\setminus\{(0,0)\}$. 
Then 
\[
\begin{cases}
D(L_{A}):=\left\{u\in D(L_{\max})\;;\;\exists C\in\C\ \text{s.t.}\ 
\displaystyle
\lim_{r\downarrow 0}\left|r^{-\frac{1}{2}}u(r)
-
C\left(a_1r^{i\nu}+a_2r^{-i\nu}\right)
\right|=0\right\}, 
\\[10pt]
L_{A}u=Lu.
\end{cases}
\]
\end{definition}

\begin{remark}\label{meaning}
If $\tilde{L}$ satisfies
$L_{\min}\subset \tilde{L}\subset L_{\max}$
and $\rho(\tilde{L})\neq \emptyset$,  by Lemma \ref{intermediate}  there exists a pair $A=(a_1, a_2)\in\C^2\setminus\{(0,0)\}$ 
such that  $\tilde{L}$ coincides with $L_{A}$. 
Moreover, if $a_1'=c a_1$ and $a_2'=c a_2$ for some $c \in \C\setminus\{0\}$, then $L_{A}=L_{A'}$. 
This implies that the map 
\[
A\in\C P_1\mapsto L_{A}\in 
\{\tilde{L}\;;\;L_{\min}\subset \tilde L\subset L_{\max}\ \&\ \rho(\tilde{L})\neq \emptyset\}
\] 
is well-defined and one to one, where $\C P_1$ denotes the Riemann sphere 
(or the one-dimensional complex projective space). 
Note that it is known in a field of mathematical physics that 
there exists a bijective map  
\[
\R P_1(\cong S^1)\to 
\{\tilde{L}\;;\;L_{\min}\subset \tilde L\subset L_{\max}\ \&\ \text{$\tilde{L}$ is selfadjoint}\}.
\] 
See Proposition \ref{adjoint} for more explanation.
\end{remark}
\noindent In order to compute the spectrum of $L_A$ we need the following preliminary result.

\begin{lemma} \label{spec}
Let $\omega =\mu e^{i\xi}\in \C_+$, $|\xi| <\pi/2$. Then $(\omega^2+L_A)$ is invertible if and only if $\varphi_{\omega,0} \notin D(L_A)$.
\end{lemma}
\begin{proof} Let us assume that  $\varphi_{\omega,0} \notin D(L_A)$ so that $\omega^2+L_A$ is injective. By \eqref{phi0.zero} this is equivalent to saying that
\begin{equation} \label{determinant}
\begin{vmatrix}
 \alpha\mu^{i\nu} e^{-\xi\nu} &
  \overline{\alpha}\mu^{-i\nu}e^{\xi\nu} \\
a_1 & a_2
\end{vmatrix}
\neq 0
\end{equation}
Let $f \in L^2(\R_+)$ and $u=c_0(f)+T_\omega f$, where  $c_0(f)$ is defined in (\ref{defc}). Then (\ref{limit}) holds 
$u \in D(L_B)$ where $B=(b_1,b_2)$ and
$$b_1=(c+W(\omega)^{-1})\alpha\mu^{i\nu}e^{-\xi \nu} \qquad b_2=(c-W(\omega)^{-1})\overline{\alpha}\mu^{i\nu}e^{\xi \nu}.
$$
The system $b_1=\kappa a_1, b_2=\kappa a_2$ has a unique solution $(c,\kappa)$ because of (\ref{determinant}). With this choice, $u \in D(L_B)=D(L_A)$ and $(\omega^2+L_A)^{-1}f=c_0(f)+T_\omega f$ is bounded because of (\ref{defc}) and Lemma \ref{green}.
\end{proof}
\noindent 
To formulate the main theorem of this paper we introduce the set
\begin{align} \label{defS}
S(\kappa)
&=\left \{-\rho e^{i\theta} \in \C: \rho^{-i\nu}e^{\theta \nu}
=\kappa e^{2i\eta}\right\}
\\
\nonumber
&=\left\{-\rho_j e^{i\theta} \in \C: \theta=\frac{ \log|\kappa|}{\nu},\  \rho_j=e^{\frac{\eta+2j\pi}{\nu}}\ ,\  j\in \Z\right\},
\end{align} 
where $\kappa \in \C\setminus \{0\}$ and
$\alpha =|\alpha|e^{i\eta} $ is defined in (\ref{defalpha}). 
Note that $S(\kappa)$ consists of double sequence $\{(z_j), j \in \Z\}$ 
lying  on the half line $\{z=-\rho e^{i\theta}\}$, 
such that $|z_j |\to \infty$ as $j \to +\infty$ 
and $|z_j| \to 0$ as $j \to -\infty$. 
The above angle $\theta$ is independent of $\alpha$ 
and the moduli of the points $z_j$ depend only on $\nu$ and  $\eta=\arg (\alpha)$. 
From (\ref{defalpha}) we see that $\eta \to \pi/2$ as $\nu \to 0$ and, 
using \cite[6.1.44, p.257]{AS}, 
\[\eta = -\nu \log \nu+(1-\log 2)\nu+\pi/4+o(1)\] 
as $\nu \to +\infty$.

\begin{theorem}\label{LA.spec}
The following assertions hold
\begin{itemize}
\item[(i)] Assume $a_1\neq 0$, $a_2\neq 0$ and let $\kappa=\frac{a_1}{a_2}$. If 
\begin{equation}\label{rg1}
|\kappa|\in\left(e^{-\nu\pi} , e^{\nu\pi}\right), 
\end{equation}
then 
\begin{align*}
\sigma(L_{A})=
[0,\infty)\cup S(\kappa).
\end{align*}
Moreover, $S(\kappa)$ coincides with the set 
of all eigenvalues of $L_A$.
\item[(ii)] If $A$ does not satisfy condition in (i), then
\begin{align*}
\sigma(L_{A})=[0,\infty).
\end{align*}
\end{itemize}
\end{theorem}

\begin{proof} 
Lemma \ref{easy} yields $[0,\infty[\subset \sigma (L_A)$. 
If $\omega =\mu e^{i\xi}\in \C_+$, $|\xi| <\pi/2$, 
Lemma \ref{spec} says that $\lambda=-\omega^2 \in \sigma(L_A)$ 
if and only if $\varphi_{\omega,0}\in D(L_A)$. By (\ref{determinant}) this happens  
if and only if 
\begin{align}\label{eq.spec}
a_1\overline{\alpha}=a_2\alpha \mu^{2i\nu}e^{-2\xi\nu}
\end{align}
or $\lambda \in S(\kappa, \alpha)$.
Since $|2\xi |<\pi$ this equation can be satisfied only when (\ref{rg1}) holds. Finally, the assertion concerning the eigenvalues follow from Lemmas \ref{behavior3}, \ref{spec}.
\end{proof}

\noindent  Finally, we characterize the adjoint of $L_A$.

\begin{proposition} \label{adjoint} 
Let $A=(a_1,a_2) \in \C^2 \setminus \{(0,0)\}$. 
Then $(L_A)^*=L_B$ where $B=(b_1,b_2)$ and $b_1=\overline{a}_2$, $b_2=\overline{a}_1$. 
$L_A$ is selfadjoint if and only if $|a_1|=|a_2|$.
\end{proposition}
\begin{proof} 
Theorem \ref{LA.spec} yields the existence of $\omega>0$ such that $\omega^2+L_A$ is invertible. 
From Lemma \ref{intermediate} we know that
$$
(\omega^2+L_A)^{-1}f=c 
\left (\int_0^\infty \varphi_{\omega,0}(s)f(s)\, ds\right )\varphi_{\omega,0}+T_\omega f
$$
for a suitable $c \in \C$ and then (\ref{a1a2}) with $\mu=\omega$ and $\xi=0$ yields 
$$
a_1=(c+W(\omega)^{-1})\alpha\omega^{i\nu} \qquad 
a_2=(c-W(\omega)^{-1})\overline{\alpha}\omega^{-i\nu}.
$$
Since, by Lemma \ref{green}, $T_\omega$ is selfadjoint we obtain
$$
(\omega^2+(L_A)^*)^{-1}f=\overline{c} 
\left (\int_0^\infty \varphi_{\omega,0}(s)f(s)\, ds\right )\varphi_{\omega,0}+T_\omega f
$$
and therefore $(L_A)^*=L_B$ where
$$
b_1=(\overline{c}+W(\omega)^{-1})\alpha\omega^{i\nu} =\overline{a}_2\qquad 
b_2=(\overline{c}-W(\omega)^{-1})\overline{\alpha}\omega^{-i\nu}=\overline{a}_1
$$
since $W(\omega)$ is purely imaginary. 
Finally, $L_A$ is selfadjoint if and only if $\overline{a}_2=ca_1$, $\overline{a}_1=ca_2$ 
for a suitable $c \in \C\setminus\{0\}$ and this happens if and only if $|a_1|=|a_2|$.
\end{proof}

\begin{remark}\label{location}
Four cases appear in the description of $\sigma(L_A)$. 

\begin{enumerate}[{Case} I. ]
\item Assume that $L_A$ is selfadjoint. 
By Proposition \ref{adjoint}, 
we have $|\kappa|=1$ and $\theta=0$. 
It follows from Theorem \ref{LA.spec} that every selfadjoint extension of $L_{\min}$ 
has infinitely many eigenvalues and its spectrum is unbounded both from above and below,  see Figure 1.

\item Next we consider the case 
\[
|\kappa|=\frac{|a_2|}{|a_1|}\in\left[e^{-\frac{\nu\pi}{2}} , e^{\frac{\nu\pi}{2}}\right].
\]
that is, $\theta\in [-\pi/2,\pi/2]$. In this case, $\rho(-L_A)$ does not contain $\overline{\C_+}\setminus\{0\}$,  
see Figure 2. 
Therefore, $-L_A$ does not generate an analytic semigroup on $L^2(\R_+)$.

\item In the case 
\[
|\kappa|=\frac{|a_2|}{|a_1|}\in 
\left(e^{-\nu\pi} , e^{\nu\pi}\right)\setminus\left[e^{-\frac{\nu\pi}{2}} , e^{\frac{\nu\pi}{2}}\right], 
\]
we have $\theta\in (-\pi,\pi)\setminus[-\pi/2,\pi/2]$ (see Figure 3). 
Hence one can expect that $-L_A$ generates an analytic semigroup on $L^2(\R_+)$. 
Indeed, we prove in Proposition \ref{gen.LA} that 
$-L_A$ generates a bounded analytic semigroup  of angle $\pi/2-|\theta|$.

\item Finally we consider the case
\[
|\kappa|=\frac{|a_2|}{|a_1|}\in 
[0, \infty]\setminus\left(e^{-\nu\pi} , e^{\nu\pi}\right).
\]
Here we use $|\kappa|=\infty$ if $a_1=0$ and $|\kappa|=0$ if $a_2=0$. 
By Theorem \ref{LA.spec} (ii) we have $\sigma(L_A)=[0,\infty)$, see Figure 4. 
As in Case III, we prove that 
$-L_A$ generates a bounded analytic semigroup on $L^2(\R_+)$ of angle $\pi/2$. 

\end{enumerate}

\begin{minipage}{0.8\textwidth}

\unitlength 0.1in
\begin{picture}( 57.2500, 47.6000)( -7.1500,-50.7500)
%
\special{pn 4}%
\special{pa 400 1400}%
\special{pa 2400 1400}%
\special{fp}%
\special{sh 1}%
\special{pa 2400 1400}%
\special{pa 2334 1380}%
\special{pa 2348 1400}%
\special{pa 2334 1420}%
\special{pa 2400 1400}%
\special{fp}%
\special{pa 1400 2400}%
\special{pa 1400 400}%
\special{fp}%
\special{sh 1}%
\special{pa 1400 400}%
\special{pa 1380 468}%
\special{pa 1400 454}%
\special{pa 1420 468}%
\special{pa 1400 400}%
\special{fp}%
%
\special{pn 20}%
\special{pa 1400 1400}%
\special{pa 2400 1400}%
\special{fp}%
\put(24.0000,-13.0000){\makebox(0,0){$\mathbb{R}$}}%
\put(15.0000,-4.0000){\makebox(0,0){$i\mathbb{R}$}}%
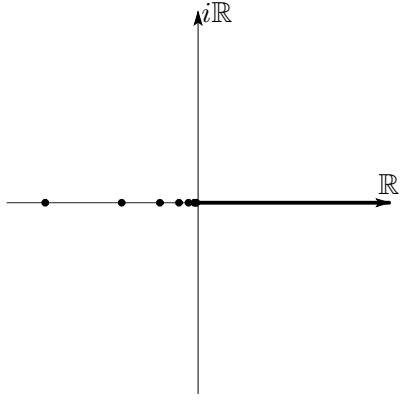
\put(14.0000,-25.5000){\makebox(0,0){Figure 1 : Selfadjoint case $\theta=0$ (Case I)}}%
%
\special{pn 20}%
\special{sh 1}%
\special{ar 600 1400 10 10 0  6.28318530717959E+0000}%
\special{sh 1}%
\special{ar 1000 1400 10 10 0  6.28318530717959E+0000}%
\special{sh 1}%
\special{ar 1200 1400 10 10 0  6.28318530717959E+0000}%
\special{sh 1}%
\special{ar 1300 1400 10 10 0  6.28318530717959E+0000}%
\special{sh 1}%
\special{ar 1350 1400 10 10 0  6.28318530717959E+0000}%
\special{sh 1}%
\special{ar 1380 1400 10 10 0  6.28318530717959E+0000}%
\special{sh 1}%
\special{ar 1390 1400 10 10 0  6.28318530717959E+0000}%
\special{sh 1}%
\special{ar 1390 1400 10 10 0  6.28318530717959E+0000}%
%
\special{pn 8}%
\special{pa 3006 1406}%
\special{pa 5006 1406}%
\special{fp}%
\special{sh 1}%
\special{pa 5006 1406}%
\special{pa 4938 1386}%
\special{pa 4952 1406}%
\special{pa 4938 1426}%
\special{pa 5006 1406}%
\special{fp}%
\special{pa 4006 2406}%
\special{pa 4006 406}%
\special{fp}%
\special{sh 1}%
\special{pa 4006 406}%
\special{pa 3986 472}%
\special{pa 4006 458}%
\special{pa 4026 472}%
\special{pa 4006 406}%
\special{fp}%
%
\special{pn 20}%
\special{pa 4006 1406}%
\special{pa 5006 1406}%
\special{fp}%
\put(50.0500,-13.0500){\makebox(0,0){$\mathbb{R}$}}%
\put(41.0500,-4.0500){\makebox(0,0){$i\mathbb{R}$}}%
\put(40.0500,-25.5500){\makebox(0,0){Figure 2 : $|\theta|\leq \pi/2$ (Case II)}}%
%
\special{pn 20}%
\special{sh 1}%
\special{ar 3370 920 10 10 0  6.28318530717959E+0000}%
\special{sh 1}%
\special{ar 3688 1162 10 10 0  6.28318530717959E+0000}%
\special{sh 1}%
\special{ar 3848 1284 10 10 0  6.28318530717959E+0000}%
\special{sh 1}%
\special{ar 3928 1344 10 10 0  6.28318530717959E+0000}%
\special{sh 1}%
\special{ar 3968 1374 10 10 0  6.28318530717959E+0000}%
\special{sh 1}%
\special{ar 3992 1392 10 10 0  6.28318530717959E+0000}%
\special{sh 1}%
\special{ar 4000 1398 10 10 0  6.28318530717959E+0000}%
\special{sh 1}%
\special{ar 4000 1398 10 10 0  6.28318530717959E+0000}%
%
\special{pn 8}%
\special{pa 406 4006}%
\special{pa 2406 4006}%
\special{fp}%
\special{sh 1}%
\special{pa 2406 4006}%
\special{pa 2338 3986}%
\special{pa 2352 4006}%
\special{pa 2338 4026}%
\special{pa 2406 4006}%
\special{fp}%
\special{pa 1406 5006}%
\special{pa 1406 3006}%
\special{fp}%
\special{sh 1}%
\special{pa 1406 3006}%
\special{pa 1386 3072}%
\special{pa 1406 3058}%
\special{pa 1426 3072}%
\special{pa 1406 3006}%
\special{fp}%
%
\special{pn 20}%
\special{pa 1406 4006}%
\special{pa 2406 4006}%
\special{fp}%
\put(24.0500,-39.0500){\makebox(0,0){$\mathbb{R}$}}%
\put(15.0500,-30.0500){\makebox(0,0){$i\mathbb{R}$}}%
\put(14.0500,-51.5500){\makebox(0,0){Figure 3 : $\pi/2<|\theta|<\pi$ (Case III)}}%
%
\special{pn 20}%
\special{sh 1}%
\special{ar 1870 3330 10 10 0  6.28318530717959E+0000}%
\special{sh 1}%
\special{ar 1648 3662 10 10 0  6.28318530717959E+0000}%
\special{sh 1}%
\special{ar 1536 3830 10 10 0  6.28318530717959E+0000}%
\special{sh 1}%
\special{ar 1482 3912 10 10 0  6.28318530717959E+0000}%
\special{sh 1}%
\special{ar 1454 3954 10 10 0  6.28318530717959E+0000}%
\special{sh 1}%
\special{ar 1436 3978 10 10 0  6.28318530717959E+0000}%
\special{sh 1}%
\special{ar 1432 3986 10 10 0  6.28318530717959E+0000}%
\special{sh 1}%
\special{ar 1432 3986 10 10 0  6.28318530717959E+0000}%
%
\special{pn 8}%
\special{pa 3010 4010}%
\special{pa 5010 4010}%
\special{fp}%
\special{sh 1}%
\special{pa 5010 4010}%
\special{pa 4944 3990}%
\special{pa 4958 4010}%
\special{pa 4944 4030}%
\special{pa 5010 4010}%
\special{fp}%
\special{pa 4010 5010}%
\special{pa 4010 3010}%
\special{fp}%
\special{sh 1}%
\special{pa 4010 3010}%
\special{pa 3990 3078}%
\special{pa 4010 3064}%
\special{pa 4030 3078}%
\special{pa 4010 3010}%
\special{fp}%
%
\special{pn 20}%
\special{pa 4010 4010}%
\special{pa 5010 4010}%
\special{fp}%
\put(50.1000,-39.1000){\makebox(0,0){$\mathbb{R}$}}%
\put(41.1000,-30.1000){\makebox(0,0){$i\mathbb{R}$}}%
\put(40.1000,-51.6000){\makebox(0,0){Figure 4 : $|\theta|\geq \pi$ (Case IV)}}%
%
\special{pn 8}%
\special{pa 4000 1400}%
\special{pa 3000 660}%
\special{fp}%
%
\special{pn 4}%
\special{ar 4000 1410 700 700  3.1415927 3.8000180}%
\put(32.1000,-11.5000){\makebox(0,0){$|\theta|$}}%
%
\special{pn 8}%
\special{pa 1400 4000}%
\special{pa 2070 3050}%
\special{fp}%
%
\special{pn 8}%
\special{ar 1410 4010 700 700  3.1415927 5.3099454}%
\put(10.3000,-33.1000){\makebox(0,0){$|\theta|$}}%
\end{picture}%

\end{minipage}
\end{remark}

\bigskip \bigskip

\section{Generation of analytic semigroups}
We characterize when $L_A$ generates an analytic semigroup.

\begin{theorem}\label{gen.LA}
Let $L_{A}$ be defined in Definition \ref{LA}. 
Then $-L_{A}$ generates a bounded analytic semigroup $\{T_A(z)\}$ on $L^2(\R_+)$ 
if and only if $a_1$ and $a_2$ satisfy 
\begin{equation}\label{good.rg}
|\kappa|=\frac{|a_2|}{|a_1|}
\in [0,\infty]\setminus \left[e^{-\frac{\nu\pi}{2}}, e^{\frac{\nu\pi}{2}}\right].
\end{equation}
Moreover, if $\theta=\frac{\log |\kappa|}{\nu}$, the maximal angle of analyticity $\theta_A$ of $\{T_A(z)\}$ is given by 
\[
\theta_A:=\begin{cases}
\ |\theta|-\dfrac{\pi}{2} & \text{if}\ |\kappa|
\in (e^{-\nu\pi}, e^{\nu\pi})\setminus \left[e^{-\frac{\nu\pi}{2}}, e^{\frac{\nu\pi}{2}}\right], 
\\[10pt]
\ \dfrac{\pi}{2}& \text{otherwise}.
\end{cases}
\]
\end{theorem}

\noindent Setting
\[
\Sigma(\theta):=\{z\in \C\setminus\{0\}\;;\;|{\rm Arg}\,z|<|\theta|\}. 
\]
from Theorem \ref{LA.spec}, we immediately obtain
\begin{lemma}\label{surj}
$\Sigma(\pi/2+\theta_A)\subset \rho(-L_{A})$. 
In particular, $\overline{\C}_+\setminus\{0\}\subset \rho(-L_{A})$
if and only if $a_1$ and $a_2$ satisfy \eqref{good.rg}.
\end{lemma}

\noindent To prove Theorem \eqref{gen.LA}, we use a scaling argument.
It worth noticing that if $a_1\neq 0$ and $a_2\neq 0$, then 
$D(L_{A})$ is not invariant under scaling $u(r)\mapsto u(s_0r)$ for some $s_0>0$ 
in spite of the scale invariant property of $D(L_{\min})$ and $D(L_{\max})$. 
This means that the scale symmetry of $L_A$ (with $s\in(0,\infty)$) is broken. 
However, there exists a subgroup $G$ of $(0,\infty)$ 
such that the scale symmetry of $L_A$ with $s\in G$ is still true.

\begin{lemma}\label{scaling}
For $\nu>0$, we define 
\begin{equation}\label{G}
G(\nu):=\left\{e^{\frac{m\pi}{\nu}}\;;\;m\in\Z\right\}.
\end{equation}
Assume that $a_1\neq 0$ and $a_2\neq 0$.
Then $D(L_{A})$ is invariant under the scaling $u(r)\mapsto u(sr)$ if and only if $s\in G(\nu)$.
On the other hand, if $a_1=0$ or $a_2=0$, then $D(L_{A})$ 
is invariant under the scaling $u(r)\mapsto u(sr)$ for every $s\in(0,\infty)$. 
\end{lemma}

\begin{proof} Fix $A=(a_1, a_2)$ with  $a_1\neq 0$ and $a_2\neq 0$and let $u\in D(L_{A})$ satisfy
\[
\lim_{r\downarrow 0}
\left|
  r^{-\frac{1}{2}}u(r)
  -
  C\left(
    a_1r^{i\nu}
    +
    a_2r^{-i\nu}
  \right)
\right|=0
\]
for some $C \neq 0$.
Then $u(sr)\in D(L_{A})$ if and only if  
\[
\lim_{r\downarrow 0}\left|r^{-\frac{1}{2}}u(sr)
-
C'\left(a_1r^{i\nu}+a_2r^{-i\nu}\right)
\right|=0
\]
for some $C'$. This is equivalent to saying that 
\begin{align}
\lim_{r\downarrow0}\left|C\left(a_1(sr)^{i\nu}+a_2(sr)^{-i\nu}\right)
-
C'\left(a_1r^{i\nu}+a_2r^{-i\nu}\right)\right|=0,
\end{align}
or
\begin{align}
Cs^{i\nu}=C'=Cs^{-i\nu}.
\end{align}
We obtain $\log s\in (\pi/\nu)\Z$, or equivalently, $s\in G(\nu)$. The cases $a_1=0$ or $a_2=0$ are similar.
\end{proof}

\begin{proof}[Proof of Theorem \ref{gen.LA}] 
Assume \eqref{good.rg}. 
For  $0<\ep<\theta_A$ let 
\[
\Sigma_\ep:=
\left\{\lambda \in \overline{\Sigma(\pi/2+\theta_A-\ep)}\;;\; 1 \leq |\lambda|\leq e^{\frac{2\pi}{\nu}}\right\}
\subset \rho(-L_A).
\]
Since $\Sigma_\ep$ is compact in $\C$, $\|(\lambda+L_{A})^{-1}\|$ is bounded in $\Sigma_\ep$. 
Therefore we have 
\[
\|(\lambda+L_{A})^{-1}\|\leq \frac{M_\ep}{|\lambda|}, \qquad \lambda\in\Sigma_\ep.
\]
Observe that by Lemma \ref{scaling} the dilation operator $(I_{s}u)(x):=s^{\frac{1}{2}}u(sx)$ satisfies 
$\|I_su\|_{L^2(\R_+)}=\|u\|_{L^2(\R_+)}$ and 
\begin{align}\label{scale.LA}
L_AI_s=s^2I_sL_A, \qquad s\in G(\nu).
\end{align}
Let $\lambda\in \Sigma(\pi/2+\theta_A-\ep)$. Taking $s\in G(\nu)$ as 
\begin{align}\label{choice.s}
\log s_0\in \left[-\frac{\log|\lambda|}{2}, \frac{\pi}{\nu}-\frac{\log|\lambda|}{2}\right)\cap \frac{\pi}{\nu}\Z\neq\emptyset, 
\end{align}
we see that $s_0^2\lambda\in \Sigma_\ep$, and hence, we have 
\[
\|(s_0^2\lambda+L_A)^{-1}\|\leq \frac{M_\ep}{|s_0^2\lambda|}.
\] 
Using \eqref{scale.LA} with \eqref{choice.s}, we obtain
\begin{align*}
\|(\lambda+L_A)^{-1}\|
&=\|(\lambda+s_0^{-2}I_{s_0^{-1}}L_AI_{s_0})^{-1}\|
=s_0^2\|I_{s_0^{-1}}(s_0^2\lambda+L_A)^{-1}I_{s_0}\|
\leq  \frac{s_0^2M_\ep}{|s_0^2\lambda|}
= \frac{M_\ep}{|\lambda|}.
\end{align*}
Therefore $-L_A$ generates a bounded analytic semigroup on $L^2(\R_+)$ of angle $\theta_A$. 
The optimality of $\theta_A$ follows from Lemma \ref{surj}.

\noindent On the other hand, if \eqref{good.rg} is violated, then Proposition \ref{surj} implies that 
$-L_{A}$ does not generates an analytic semigroup on $L^2(\R_+)$. 
\end{proof}

\begin{remark}
In the case $|\kappa|=e^{\frac{\nu\pi}{2}}$ or $|\kappa|=e^{-\frac{\nu\pi}{2}}$, 
we do not know whether the operator $-L_{A}$ generates a $C_0$-semigroup on $L^2(\R_+)$. 
We point out that if $-L_{A}$ generates a $C_0$-semigroup, then 
it cannot be (quasi) contractive because Hardy's inequality does not hold on $C_0^\infty (\R_+)$, since $b <-\frac14$.
\end{remark}

\section{Remarks on the $N$-dimensional case}

Here we consider the $N$-dimensional Schr\"odinger operators, $N \ge 2$,  
\[
L=-\Delta + \frac{b}{|x|^2}, \quad b \in \R
\]
As in one dimension  we define 
\begin{align*}
D(L_{\min})&:=C_0^\infty(\R^N\setminus\{0\}), 
\\
D(L_{\max})&:=\{u\in L^2(\R^N)\cap H^2_{\rm loc}(\R^N\setminus\{0\})\;;\; Lu\in L^2(\R^N)\}.
\end{align*}
Hardy's inequality
\begin{equation}
\left(\frac{N-2}{2}\right)^2\int_{\R^N}\frac{|u|^2}{|x|^2}\,dx
\leq \int_{\R^N}|\nabla u|^2\,dx, 
\quad u\in C_0^\infty(\R^N\setminus\{0\})
\end{equation} 
implies the existence 
of a nonegative selfadjoint extension of $L_{\min}$, namely the Friedrichs extension, for $b\geq -(\frac{N-2}{2})^2$. 
Therefore in this section we assume 
\begin{equation}\label{sup.Nd}
b<-\left(\frac{N-2}{2}\right)^2.
\end{equation}
Using Proposition \ref{gen.LA} we obtain the following result.

\begin{proposition}\label{Nd.gen}
Assume that \eqref{sup.Nd} holds. Then 
there exist infinitely many intermediate operators 
between $L_{\min}$ and $L_{\max}$ which are negative generators of
analytic semigroups on $L^2(\R^N)$. 
\end{proposition}

\noindent To prove Proposition \ref{Nd.gen} we expand $f \in L^2(\R^N)$ in spherical harmonics
\begin{equation*}
f=\sum_{j=0}^\infty F_{j}(G_{j}f). 
\end{equation*}
where 
$F_j:L^2(\R_+)\to L^2(\R^N)$ and $G_j:L^2(\R^N)\to L^2(\R_+)$ are defined by
\begin{align*}
F_jg(x)&:=|x|^{-\frac{N-1}{2}}g(|x|)Q_j(\omega), \quad g\in L^2(\R_+),
\\
G_jf(r)&:=r^{\frac{N-1}{2}}\int_{S^{N-1}} f(r,\omega)Q_j(\omega)\,d\omega, \quad f\in L^2(\R^N).
\end{align*}
Here $\{Q_j\;;\;j\in \N\}$ is a orthonormal basis of $L^2(S^{N-1})$ 
consisting of spherical harmonics $Q_j$  of order $n_j$. 
$Q_j$ is an eigenfunction of Laplace-Beltrami operator $\Delta_{S^{N-1}}$ 
with respect to the eigenvalue $-\lambda_j=-n_{j}(N-2+n_j)$, 
see e.g., \cite[30, Chapter IX]{V} and also \cite[Ch.\ 4, Lemma 2.18]{SWbook}. 

\begin{lemma}\label{lem.FG}
For every $j\in \N_0=\N\cup\{0\}$ the following assertions hold
\begin{itemize}
\item[(i)]   
$\|F_{j}g\|_{L^2(\R^N)}=\|g\|_{L^2(\R_+)}$ for every $g\in L^2(\R_+)$, 
$\|G_{j}f\|_{L^2(\R_+)}\leq \|f\|_{L^2(\R^N)}$ for every $f\in L^2(\R^N)$;
\item[(ii)]   $G_{j}F_{j}=1_{L^2(\R_+)}$ and $F_{j}G_{j}[D(L_{\min})]\subset D(L_{\min})$;
\item[(iii)] for every $v\in C_0^\infty(\R_+)$, 
\[
G_j[L(F_jv)](r)=-v''(r)+\frac{b_{j}}{r^{2}}v(r), 
\]
where 
\[
b_{j}:=b+\left(\frac{N-2}{2}\right)^2-\frac{1}{4}+\lambda_{n_j}.
\]
\end{itemize}
\end{lemma}

\begin{proof}
(i) and (ii) follow easily by direct computation. We only prove (iii).
Let $v\in C_c^\infty(\R_+)$. 
Observing that 
\[
L=-\frac{\pa^2 }{\pa r^2}
  -\frac{N-1}{|x|}\frac{\pa}{\pa r}
  +\frac{b}{|x|^2}
  -\frac{1}{|x|^2}\Delta_{S^{N-1}},
\] 
we deduce 
\begin{align*}
L(F_j v)(x)
&=L\left(|x|^{-\frac{N-1}{2}}v(|x|)Q(\omega)\right)
\\
&=|x|^{-\frac{N-1}{2}}\biggl[
-v''(|x|)+
\left(b+\frac{(N-1)(N-3)}{4}+\lambda_{n_j}\right)
\frac{1}{|x|^{2}}v(|x|)
\biggr]
Q(\omega).
\end{align*}
Therefore, 
\[
G_j[L(F_jv)](r)
=-v''(r)+\frac{b_j}{r^2}.
\]
\end{proof}

\begin{proof}[Proof of Proposition \ref{Nd.gen}]
If $j\in \N_0$ satisfies $b_j\geq \frac{1}{4}$, then 
from Lemma \ref{lem.FG} (ii)
$L_{j,\min}:=F_{j}L_{\min}G_{j}$ is nonnegative, and hence 
there exists a Friedrichs extension $L_{j,F}$ of $L_{j,\min}$. 
This implies that 
\begin{align*}
\|(\lambda-L_{j,F})^{-1}\|\leq \frac{1}{|\lambda|}\qquad \lambda\in \C_+.
\end{align*}
If $j\in \N$ satisfies $b_j<-\frac{1}{4}$, then 
we choose $A_{j}=(a_{1,j},a_{2,j})\in \C^2\setminus \{(0,0)\}$ satisfying \eqref{good.rg} 
with $\nu_j=\sqrt{-b_j-1/4}$. 
By Proposition \ref{gen.LA}, we have 
\begin{align*}
\|(\lambda-L_{j,A_j})^{-1}\|\leq \frac{M_j}{|\lambda|}\qquad \lambda\in \C_+.
\end{align*}
Now we define the operator $\tilde{L}$ between $L_{\min}$ and $L_{\max}$ as follows:
\begin{align*}
D(\tilde{L}):=
\left(\bigoplus_{b_j\geq-1/4}F_{j}D(L_{j,F})\right)\oplus
\left(\bigoplus_{b_j<-1/4}F_{j}D(L_{j,A_j})\right);
\end{align*}
note that $\tilde{L}\supset L_{\min}$ is verified 
by Lemma \ref{lem.FG} (ii). 
Then we see $\lambda-\tilde{L}$ is injective for every $\lambda\in \C_+$. 
In fact, if $\lambda - \tilde{L}u=0$ for $u\in D(\tilde{L})$, then 
for every $j\in \N$, by the definition of $D(\tilde{L})$ 
it follows from Lemma \ref{lem.FG} (iii) that 
$(\lambda-L_{j,F})u_j=0$ with $u_j:=G_ju\in D(L_{j,F})$ when $b_j\geq-1/4$ 
and 
$(\lambda-L_{j,A_j})u_j=0$ with $u_j:=G_ju\in D(L_{j,A_j})$ when $b_j<-1/4$.
This implies that $u_j=0$ for every $j\in \N$, hence $u=\sum_{j\in\N}F_ju_j=0$.

\noindent Moreover, for every $f\in L^2(\R^N)$, we have $f=\lambda u -\tilde{L} u$, where we set 
\[
u:=
\sum_{b_j\geq-1/4}F_{j}(\lambda-L_{j,F})^{-1}G_{j}f
+
\sum_{b_j<-1/4}F_{j}(\lambda-L_{j,A_j})^{-1}G_{j}f\in D(\tilde{L}).
\]
Since the $\{j\in \N\;;\;b_j<-1/4\}$ is finite, 
$\tilde{M}:=\max\{M_j\;;\;b_j<-1/4\}$ is also finite.
Hence it follows from Lemma \ref{lem.FG} (i) that
for every $\lambda\in \C_+$, 
\begin{align*}
\|u\|_{L^2(\R^N)}^2
&=
\sum_{b_j\geq-1/4}\|F_{j}(\lambda-L_{j,F})^{-1}G_{j}f\|_{L^2(\R^N)}^2
+
\sum_{b_j<-1/4}\|F_{j}(\lambda-L_{j,A_j})^{-1}G_{j}f\|_{L^2(\R^N)}^2
\\
&=
\sum_{b_j\geq-1/4}\|(\lambda-L_{j,F})^{-1}G_{j}f\|_{L^2(\R_+)}^2
+
\sum_{b_j<-1/4}\|(\lambda-L_{j,A_j})^{-1}G_{j}f\|_{L^2(\R_+)}^2
\\
&\leq
\sum_{b_j\geq-1/4}\frac{1}{|\lambda|}\|G_{j}f\|_{L^2(\R_+)}^2
+
\sum_{b_j<-1/4}\frac{M_j}{|\lambda|}\|G_{j}f\|_{L^2(\R_+)}^2
\\
&\leq
\frac{\tilde{M}}{|\lambda|} 
\left(\sum_{b_j\geq-1/4}\|F_{j}G_{j}f\|_{L^2(\R^N)}^2
+
\sum_{b_j<-1/4}\|F_{j}G_{j}f\|_{L^2(\R^N)}^2\right)
\\&=
\frac{\tilde{M}}{|\lambda|} \|f\|_{L^2(\R^N)}^2.
\end{align*}
Therefore $\tilde{L}$ is closed, $\C_+\subset \rho(-\tilde{L})$ and 
\[
\|(\lambda-\tilde{L})^{-1}\|\leq \frac{\tilde{M}}{|\lambda|}.
\] 
This implies that $-\tilde{L}$ generates a bounded analytic semigroup on $L^2(\R^N)$.
Since we can choose all of $A_{j}$ satisfying \eqref{good.rg}, 
we can produce infinitely many (negative) generators between $L_{\min}$ and $L_{\max}$.
\end{proof}

{\small
\begin{thebibliography}{26}

 \bibitem{AS}
      M. Abramowitz, I.A. Stegun, 
      ``Handbook of mathematical functions with formulas, graphs, and mathematical tables,'' 
      National Bureau of Standards Applied Mathematics Series {\bf 55}. 
      Washington, D.C. 1964. 

 \bibitem{BG84}
      P. Baras, J.A.\ Goldstein, 
      {\it The heat equation with a singular potential},  
      Trans.\ Amer.\ Math.\ Soc.\ {\bf 284} (1984), 
      121--139. 

 \bibitem{BR}
	 G.~Birkhoff, G. ~Rota, 
     ``Ordinary Differential Equations,"
     Wiley, 1989.

 \bibitem{CM99}
      X. Cabr\'e, Y. Martel, 
      {\it Existence versus instantaneous blowup for linear heat equations with singular potentials}
      C.\ R.\ Acad.\ Sci.\ Paris\ S\'er.\ I Math.\ {\bf 329} (1999), 973--978.

   \bibitem{GTV10}
     D.M.\ Gitman, I.V. Tyutin, B.L.\ Voronov, 
         {\it Self-adjoint extensions and spectral analysis in the Calogero problem},  
     J.\ Phys.\ A\ {\bf 43} (2010), 145205, 34 pp. 

  \bibitem{Ga08}
     V.A.\ Galaktionov, 
     {\it On nonexistence of Baras-Goldstein type without positivity assumptions for singular linear 
     and nonlinear parabolic equations}, 
     Tr.\ Mat.\ Inst.\ Steklova {\bf 260} (2008), 
     Teor.\ Funkts.\ i Nelinein.\ Uravn.\ v Chastn.\ Proizvodn., 130--150; 
     translation in Proc.\ Steklov Inst.\ Math.\ {\bf 260} (2008), 123--143.
   \bibitem{GJ01}
     J.A. Goldstein, Q.S. Zhang, 
     {\it On a degenerate heat equation with a singular potential}, 
     J.\ Funct.\ Anal.\ {\bf 186} (2001), 342--359.  

   \bibitem{GJ02}
     J.A. Goldstein, Q.S. Zhang, 
     {\it Linear parabolic equations with strong singular potentials}, 
     Trans.\ Amer.\ Math.\ Soc.\ {\bf 355} (2003), 197--211.   

   \bibitem{Ma03}
    C. Marchi, 
      {\it The Cauchy problem for the heat equation with a singular potential}, 
      Differential and Integral Equations {\bf 16} (2003), 1065--1081. 

  \bibitem{MOSS-1}
     G. Metafune, N. Okazawa, M. Sobajima, C. Spina, 
     {\it Scale invariant elliptic operators with singular coefficents}, 
      preprint.

  \bibitem{MS_bhv}
     G. Metafune, M. Sobajima, 
	 {\it An elementary proof of asymptotic behavior of solutions of $u''=Vu$}, 
     preprint. 

\bibitem{RS2}
           M.\ Reed, B.\ Simon, 
           ``Methods of modern mathematical physics. II. Fourier analysis, self-adjointness,'' 
           Academic Press, New York-London, 1975.

\bibitem{SWbook}
     E M. Stein, G. Weiss,
     ``Introduction to Fourier analysis on euclidean spaces, ''
     Princeton University Press, 1971.

  \bibitem{VZ00}
     J.L. Vazquez, E. Zuazua, 
     {\it The Hardy inequality and the asymptotic behaviour of the heat equation 
     with an inverse-square potential}, 
     J.\ Funct.\ Anal.\ {\bf 173} (2000), 103--153. 
     
  \bibitem{V}
     N.\ Ja.\ Vilenkin, 
     ``Fonctions sp\'eciales et th\'eorie de la repr\'esentation des groupes,'' 
     Dunod Paris, 1969; 
     English translation, 
     ``Special functions and the theory of group representation,''
     Trans.\ Math.\ Monographs {\bf 22}, Amer.\ Math.\ Soc., Providence, R.I., 1968.
\end{thebibliography}
}
\end{document}